\documentclass[12pt,a4paper]{article}

\usepackage[utf8]{inputenc}
\usepackage[english]{babel}
\usepackage{mathtools} 
\usepackage{amsfonts} 
\usepackage{amsmath}
\usepackage{amssymb}
\usepackage{amsthm}
\usepackage{mathrsfs}
\usepackage{microtype}
\usepackage{marvosym}
\usepackage{wasysym} 
\usepackage{stmaryrd}
\usepackage{dsfont}
\usepackage{enumerate}
\usepackage{multicol,multirow}
\usepackage{graphicx}
\usepackage{xcolor}
\usepackage{authblk}
\usepackage{hyperref}
\hypersetup{
	colorlinks=false,
	linkcolor=blue,
	urlcolor=red,
	linktoc=all}

\usepackage{tikz} 
\usetikzlibrary{decorations.pathreplacing}
\usetikzlibrary{decorations.pathmorphing}
\usetikzlibrary{calc}
\usetikzlibrary{decorations.pathreplacing}
\usetikzlibrary{decorations.pathmorphing}
\usetikzlibrary{arrows}
\usetikzlibrary{patterns,fadings}
\usepackage[all,cmtip]{xy}

\newcommand{\eps}{\text{\Large{$\varepsilon$}}}
\newcommand{\oneop}{\mathds{1}}

\newcommand{\C}{\mathcal{C}}

\newcommand{\A}{\mathcal{A}}

\newcommand{\RR}{\mathbb{R}}

\newcommand{\CC}{\mathbb{C}}

\DeclareMathOperator{\Hom}{Hom}

\DeclareMathOperator{\id}{id}
\DeclareMathOperator{\Id}{Id}
\DeclareMathOperator{\Tr}{Tr}

\newcommand{\op}{\mathrm{op}}

\newcommand{\lqq}{\lq\lq}

\usepackage{xspace}
\DeclareRobustCommand{\eg}{e.g.\@\xspace}
\DeclareRobustCommand{\cf}{cf.\@\xspace}
\DeclareRobustCommand{\ie}{i.e.\@\xspace}
\DeclareRobustCommand{\p}{p.\@\xspace}
\DeclareRobustCommand{\Sec}{Sec.\@\xspace}
\DeclareRobustCommand{\Prop}{Prop.\@\xspace}
\DeclareRobustCommand{\Lem}{Lem.\@\xspace}
\DeclareRobustCommand{\Cor}{Cor.\@\xspace}
\DeclareRobustCommand{\Thm}{Thm.\@\xspace}

\DeclareRobustCommand{\Def}{Def.\@\xspace}
\DeclareRobustCommand{\Rmk}{Rmk.\@\xspace}
\DeclareRobustCommand{\Eq}{Eq.\@\xspace}

\DeclareRobustCommand{\rhs}{r.h.s.\@\xspace}

\makeatletter
\DeclareRobustCommand{\etc}{%
    \@ifnextchar{.}%
        {etc}%
        {etc.\@\xspace}%
}
\makeatother

\def\u1net{{\A_\RR}}

\theoremstyle{plain}
\newtheorem{theorem}{Theorem}[section]

\newtheorem{proposition}[theorem]{Proposition}

\theoremstyle{definition}
\newtheorem{definition}[theorem]{Definition}

\theoremstyle{remark}

\newtheorem{remark}[theorem]{Remark}

 \newcommand{\tikzmath}[2][\squarescale]
	{\vcenter{\hbox{
	\begin{tikzpicture}
	[scale=#1]#2
	\end{tikzpicture}}}}

\begin{document}

\title{Bantay's trace in Unitary Modular Tensor Categories}


\author[1]{Luca Giorgetti}
\author[2]{Karl-Henning Rehren}
\affil[1]{\small Dipartimento di Matematica, Universit\`a di Roma Tor Vergata\\
Via della Ricerca Scientifica, 1, I-00133 Roma, Italy\\
{\tt giorgett@mat.uniroma2.it}}
\affil[2]{Institut f\"ur Theoretische Physik, Georg-August-Universit\"at G\"ottingen\\
Friedrich-Hund-Platz, 1, D-37077 G\"ottingen, Germany\\
{\tt rehren@theorie.physik.uni-goettingen.de}}

\date{}

\maketitle

\begin{abstract}
We give a proof of a formula for the trace of self-braidings (in an arbitrary channel) in UMTCs which first appeared in the context of rational conformal field theories (CFTs) \cite{Ban97}. The trace is another invariant for UMTCs which depends only on modular data, and contains the expression of the Frobenius-Schur indicator \cite{NgSc07} as a special case. Furthermore, we discuss some applications of the trace formula to the realizability problem of modular data and to the classification of UMTCs. 
\end{abstract}

\section{Introduction}

The formula for the Frobenius-Schur indicator in rational CFT appears in the work of \cite[\Eq
(1)]{Ban97} without proof. The author only shows that the possible
values are $0$, $\pm 1$ and that they fit with self-conjugacy, reality
or pseudo-reality properties of primary fields. Afterwards, \cite[\Sec
7]{NgSc07} derived the same formula as a special case of \lq\lq higher
degree" indicators, in the more general context of modular tensor
categories (MTCs). We give here a proof, in the case of unitary MTCs (UMTCs), which has the advantage of being simpler and closer to the lines of Bantay. In doing so, we also prove a more general formula, \cf \cite[\Eq (3)]{Ban97}, which expresses the trace of the self-braidings $\eps_{a,a}$ in an \emph{arbitrary} \lq\lq channel" $c\prec a\times a$ (not only $c = \id$ as one needs for the determination of the Frobenius-Schur indicator).  All these formulas have the remarkable property of depending \emph{only} on modular data. In particular, using the trace of self-braidings, we show that the braiding in a UMTC is uniquely determined when the underlying unitary tensor (fusion) structure (UFTC) and the modular data are given.

\section{Bantay's trace}

Let $\C$ be a UMTC. We denote by $n+1$ the rank of $\C$, by $\Delta$
and $N$ respectively its spectrum (set of unitary isomorphism classes
of irreducible objects) and fusion rules, and by $S$ and $T$ its
modular matrices. The numerical invariants $\{n, \Delta, N, S, T\}$
are the \emph{modular data} of $\C$, and recall that they can be
recovered by either $\{S, T\}$ or $\{N, T\}$ alone, due to the
constraints imposed by modularity. In particular, the dimensions $d_i$
and phases $\omega_i$ of the sectors $[a_i]$ are given by
$d_i=S_{0,i}/S_{0,0}$ and $\omega_i = T_{i,i}/T_{0,0}$. 
We refer to \cite{BaKi00}, \cite{Mue10tour}, \cite{EGNO15} for more explanations of terminology and results on UMTCs, and to \cite{GioPhD} for our precise conventions in the widely-used string diagrammatics employed below.

\begin{definition}
For any $[a_i]\in\Delta$ and $[a_k]\in\Delta$ let $m := N^k_{i,i}$, \ie, $m$ is the multiplicity of $[a_k]$ in $[a_i] \times [a_i]$, and $\eps_{a_i,a_i}$ the self-braiding of $a_i$ with itself. Define
$$\tau_{k,i} \equiv \Tr_{a_k}(\eps_{a_i,a_i}) := \sum_{e=1,\ldots,m} (t^e_k)^*\circ \eps_{a_i,a_i} \circ t^e_k$$
where $t^e_k$, $e=1,\ldots, m$ is a linear basis of $\Hom_{\C}(a_k,a_i \times a_i)$ of orthonormal isometries in the sense that $(t^e_k)^*\circ t^{e'}_k = \delta_{e,e'}1_k$. If $m=0$, then $\tau_{k,i} := 0$.

The number $\tau_{k,i}$ does not depend on the choice of objects
$a_i\in[a_i]$, $a_k\in[a_k]$ in their equivalence classes, nor on that of the basis of isometries, by naturality of the braiding and by the trace property of left inverses \cite[\Lem 3.7]{LoRo97}. Hence it defines another invariant for the UMTC $\C$.

In particular, for $k=0$, $\nu_i:=\omega_i\tau_{0,i}$ is the
Frobenius-Schur indicator, which takes the values $0$, $+1$ or $-1$
(\cf \cite{Ban97} and \Sec \ref{subsec:realiz}) respectively if $[a_i]$ is
non-self-conjugate, real or pseudo-real in $\C$. (See \cite[\p
139]{LoRo97} for the notion of reality and pseudo-reality of objects
in unitary tensor categories. Another terminology is complex, real,
and quaternionic.)
\end{definition}

\begin{remark}
Using the trace property of left inverses and naturality of the braiding (\cf proof of \cite[\Prop 2.4]{Mue00}), one can show that $\tau_{k,i} = \tau_{\overline k,\overline i}$. 
\end{remark}

\begin{proposition}\label{prop:bantaytrace}
Let $\C$ be a UMTC. The invariants $\tau_{k,i}$ can be expressed in terms of modular data, namely
\begin{equation}\label{eq:bantaytrace}
\tau_{k,i}
%
%
	\,=\, \omega_i^{-1} \sum_{r,s\in\Delta} \overline{S_{r,k}} S_{s,0} N^{i}_{r,s} \frac{\omega_s^2}{\omega_r^2}
\end{equation}
for every $i,k = 0,\ldots,n$.

In particular, for every $i=0,\ldots,n$ we have
\begin{equation}\label{eq:bantayFSindicator}
\nu_i = \sum_{r,s\in\Delta}  S_{r,0} S_{s,0} N^{i}_{r,s} \frac{\omega_r^2}{\omega_s^2}\,.
\end{equation}
\end{proposition}

\begin{remark}\label{rmk:conjugationmissing}
Notice that equation (\ref{eq:bantaytrace}) differs from \cite[\Eq (3)]{Ban97} by a complex conjugation of $S_{r,k}$. This is presumably due to a different convention: while our $S$ matrix satisfies $(ST)^3=C$, see \cite[\Eq (5.19--21)]{Reh90}, its complex conjugate $\overline{S} = S^{-1}$ satisfies $(\overline{S} T)^3=\oneop$, and Bantay doesn't specify his choice.

In the special case of equation (\ref{eq:bantayFSindicator}), which coincides with Bantay's expression \cite[\Eq (1)]{Ban97} for the Frobenius-Schur indicator, the complex conjugation is irrelevant because $k=0$ is self-conjugate. 
\end{remark}

\begin{proof}
The following argument makes clear the advantages of using the string-diagrammatical notation, indeed the proof written with usual sums, compositions and products of morphism would be (to us) almost unreadable.

Assume first that $m>0$. Using string diagrammatics, the trace is
$$\tau_{k,i} = 
\sum_{e=1,\ldots,m} 
\tikzmath{
	\draw[ultra thick] (2.5,9) -- (2.5,15);
	\draw[ultra thick] (2.5,-9) -- (2.5,-15);
	\draw (5,11.5) node {$e$};
	\draw (5.1,-11.4) node {$e$};
	\draw (2.5,19) node {$k$};
	\draw (2.5,-19) node {$k$};

	\draw[ultra thick] (0,5) -- (0,5) .. controls (0,10) and (5,10) .. (5,5) -- (5,5);
	\draw[ultra thick] (0,-5) -- (0,-5).. controls (0,-10) and (5,-10) .. (5,-5) -- (5,-5);
	\draw (-2,5) node {$i$};
	\draw (7,5) node {$i$};

	\draw[ultra thick] (5,5) -- (5,5) .. controls (5,0) and (0,0) .. (0,-5) -- (0,-5);
	\fill[color=white] (2.5,0) circle (2); 
	\draw[ultra thick] (0,5) -- (0,5) .. controls (0,0) and (5,0) .. (5,-5) -- (5,-5);
	}.
$$
We write
$$
\sum_{e} 
\tikzmath{
	\draw[ultra thick] (2.5,9) -- (2.5,15);
	\draw[ultra thick] (2.5,-9) -- (2.5,-15);
	\draw (5,11.5) node {$e$};
	\draw (5.1,-11.4) node {$e$};
	\draw (2.5,19) node {$k$};
	\draw (2.5,-19) node {$k$};

	\draw[ultra thick] (0,5) -- (0,5) .. controls (0,10) and (5,10) .. (5,5) -- (5,5);
	\draw[ultra thick] (0,-5) -- (0,-5).. controls (0,-10) and (5,-10) .. (5,-5) -- (5,-5);
	\draw (-2,5) node {$i$};
	\draw (7,5) node {$i$};

	\draw[ultra thick] (5,5) -- (5,5) .. controls (5,0) and (0,0) .. (0,-5) -- (0,-5);
	\fill[color=white] (2.5,0) circle (2); 
	\draw[ultra thick] (0,5) -- (0,5) .. controls (0,0) and (5,0) .. (5,-5) -- (5,-5);
	}
=
\sum_{e} \frac{1}{d_k} 
\tikzmath{
	\draw[ultra thick] (2.5,9) -- (2.5,12);
	\draw[ultra thick] (2.5,-9) -- (2.5,-12);
	\draw (5,11.5) node {$e$};
	\draw (5.1,-11.4) node {$e$};
	\draw (-2,5) node {$i$};
	\draw (7,5) node {$i$};
	
	\draw[ultra thick] (0,5) -- (0,5) .. controls (0,10) and (5,10) .. (5,5) -- (5,5);
	\draw[ultra thick] (0,-5) -- (0,-5).. controls (0,-10) and (5,-10) .. (5,-5) -- (5,-5);
	
	\draw[ultra thick] (5,5) -- (5,5) .. controls (5,0) and (0,0) .. (0,-5) -- (0,-5);
	\fill[color=white] (2.5,0) circle (2); 
	\draw[ultra thick] (0,5) -- (0,5) .. controls (0,0) and (5,0) .. (5,-5) -- (5,-5);
	
	\draw[ultra thick] (2.5,12) -- (2.5,12) .. controls (2.5,22) and (12.5,22) .. (12.5,12) -- (12.5,12);
	\draw[ultra thick] (12.5,12) -- (12.5,-12);
	\draw[ultra thick] (2.5,-12) -- (2.5,-12) .. controls (2.5,-22) and (12.5,-22) .. (12.5,-12) -- (12.5,-12);
	\draw (0.5,19) node {$k$};
	\draw (15.5,0) node {$\overline k$};
	}
=
\sum_{e} \frac{1}{d_k} 
\tikzmath{
	\draw[ultra thick] (5,5) -- (5,5) .. controls (5,0) and (0,0) .. (0,-5) -- (0,-5);
	\fill[color=white] (2.5,0) circle (2); 
	\draw[ultra thick] (0,5) -- (0,5) .. controls (0,0) and (5,0) .. (5,-5) -- (5,-5);
	\draw (-2,5) node {$i$};
	\draw (7,5) node {$i$};
	
	\draw[ultra thick] (0,5) -- (0,10) .. controls (0,32) and (22.5,32) .. (22.5,10) -- (22.5,10);
	\draw[ultra thick] (5,5) -- (5,10) .. controls (5,25) and (17.5,25) .. (17.5,10) -- (17.5,10);
	\draw[ultra thick] (0,-5) -- (0,-10) .. controls (0,-32) and (22.5,-32) .. (22.5,-10) -- (22.5,-10);
	\draw[ultra thick] (5,-5) -- (5,-10) .. controls (5,-25) and (17.5,-25) .. (17.5,-10) -- (17.5,-10);
	
	\draw[ultra thick] (17.5,10) -- (17.5,10) .. controls (17.5,5) and (22.5,5) .. (22.5,10) -- (22.5,10);
	\draw[ultra thick] (20,6) -- (20,-6);
	\draw (23,0) node {$\overline k$};
	\draw (17,4) node {$e$};
	\draw (17,-4) node {$e$};

	\draw[ultra thick] (17.5,-10) -- (17.5,-10) .. controls (17.5,-5) and (22.5,-5) .. (22.5,-10) -- (22.5,-10);
	} 
$$
by the trace property. Invertibility of the $S$ matrix (equivalent to \emph{modularity} of $\C$ by \cite[\Sec 5]{Reh90}), or better $S^2 = C$, gives $\sum_{\beta=0,\ldots,n} S_{\alpha,\beta} S_{\beta,k} = \delta_{\alpha,\overline k}$ for every $\alpha = 0,\ldots, n$, hence the previous line can be rewritten as
$$
= \sum_{f,\alpha,\beta} \frac{1}{d_{\overline\alpha}} S_{\alpha,\beta} S_{\beta,k} 
\tikzmath{
	\draw[ultra thick] (5,5) -- (5,5) .. controls (5,0) and (0,0) .. (0,-5) -- (0,-5);
	\fill[color=white] (2.5,0) circle (2); 
	\draw[ultra thick] (0,5) -- (0,5) .. controls (0,0) and (5,0) .. (5,-5) -- (5,-5);
	\draw (-2,5) node {$i$};
	\draw (7,5) node {$i$};
	
	\draw[ultra thick] (0,5) -- (0,10) .. controls (0,32) and (22.5,32) .. (22.5,10) -- (22.5,10);
	\draw[ultra thick] (5,5) -- (5,10) .. controls (5,25) and (17.5,25) .. (17.5,10) -- (17.5,10);
	\draw[ultra thick] (0,-5) -- (0,-10) .. controls (0,-32) and (22.5,-32) .. (22.5,-10) -- (22.5,-10);
	\draw[ultra thick] (5,-5) -- (5,-10) .. controls (5,-25) and (17.5,-25) .. (17.5,-10) -- (17.5,-10);
	
	\draw[ultra thick] (17.5,10) -- (17.5,10) .. controls (17.5,5) and (22.5,5) .. (22.5,10) -- (22.5,10);
	\draw[ultra thick] (20,6) -- (20,-6);
	\draw (23,0) node {$\alpha$};
	\draw (17,3) node {$f$};
	\draw (17,-5) node {$f$};

	\draw[ultra thick] (17.5,-10) -- (17.5,-10) .. controls (17.5,-5) and (22.5,-5) .. (22.5,-10) -- (22.5,-10);
	}
= \sum_{f,\alpha,\beta} \frac{1}{d_{\overline{\alpha}}} \frac{d_\alpha}{|\sigma|} S_{\beta,k}
\tikzmath{
	\draw[ultra thick] (5,5) -- (5,5) .. controls (5,0) and (0,0) .. (0,-5) -- (0,-5);
	\fill[color=white] (2.5,0) circle (2); 
	\draw[ultra thick] (0,5) -- (0,5) .. controls (0,0) and (5,0) .. (5,-5) -- (5,-5);
	\draw (-2,5) node {$i$};
	\draw (7,5) node {$i$};
	\draw (24,16) node {$f$};
	\draw (23,-12) node {$f$};	
	
	\draw[ultra thick] (0,5) -- (0,10) .. controls (0,32) and (22.5,32) .. (22.5,20) -- (22.5,20);
	\draw[ultra thick] (5,5) -- (5,10) .. controls (5,25) and (17.5,25) .. (17.5,20) -- (17.5,20);
	\draw[ultra thick] (0,-5) -- (0,-10) .. controls (0,-32) and (22.5,-32) .. (22.5,-20) -- (22.5,-20);
	\draw[ultra thick] (5,-5) -- (5,-10) .. controls (5,-25) and (17.5,-25) .. (17.5,-20) -- (17.5,-20);
	
	\draw[ultra thick] (17.5,20) -- (17.5,20) .. controls (17.5,15) and (22.5,15) .. (22.5,20) -- (22.5,20);
	\draw[ultra thick] (20,16) -- (20,-16);
	\draw (17,3) node {$\alpha$};
	\draw (28,4) node {$\beta$};	
	\draw[ultra thick] (17.5,-20) -- (17.5,-20) .. controls (17.5,-15) and (22.5,-15) .. (22.5,-20) -- (22.5,-20);
	
	\draw[ultra thick] (12,-5) -- (12,-5) .. controls (10,-2) and (10,2) .. (12,5) -- (12,5);
	\draw[ultra thick] (25,3) -- (25,3) .. controls (25,8) and (19,15) .. (12,5) -- (12,5);
	\fill[color=white] (20,10) circle (2.5); 
	\draw[ultra thick] (20,16) -- (20,0);
	\fill[color=white] (20,-5) circle (2.5); 
	\draw[ultra thick] (25,3) -- (25,3) .. controls (25,-2) and (15,-12) .. (12,-5) -- (12,-5);
	}
$$
where $t_{\alpha}^f$, for $\alpha = 0,\ldots, n$ and $f=1,\ldots,N^{\alpha}_{\overline i \overline i}$, runs over orthonormal bases of isometries in $\Hom_\C(a_\alpha,a_{\overline i} \times a_{\overline i})$, whenever $N^{\alpha}_{\overline i \overline i} > 0$, which are in addition mutually orthogonal. The \rhs is obtained by definition of $S_{\alpha,\beta} = |\sigma|^{-1} Y_{\alpha,\beta}$, opening the $\alpha$-ring up to multiplication with $d_\alpha$. This previous insertion procedure is usually referred to as \lq\lq killing-ring'', after \cite{BEK99}. Notice also that $d_\alpha = d_{\overline \alpha}$. By naturality and multiplicativity of the braiding we get
$$
= \sum_{f,\alpha,\beta} \frac{1}{|\sigma|} S_{\beta,k}
\tikzmath{
	\draw[ultra thick] (5,5) -- (5,5) .. controls (5,0) and (0,0) .. (0,-5) -- (0,-5);
	\fill[color=white] (2.5,0) circle (2); 
	\draw[ultra thick] (0,5) -- (0,5) .. controls (0,0) and (5,0) .. (5,-5) -- (5,-5);
	\draw (-2,5) node {$i$};
	\draw (7,5) node {$i$};
	\draw (17,13) node {$f$};
	\draw (23,13) node {$\alpha$};
	\draw (30,-7) node {$\beta$};	

	\draw[ultra thick] (0,5) -- (0,10) .. controls (0,32) and (22.5,32) .. (22.5,20) -- (22.5,20);
	\draw[ultra thick] (5,5) -- (5,10) .. controls (5,25) and (17.5,25) .. (17.5,20) -- (17.5,20);
	\draw[ultra thick] (0,-5) -- (0,-10) .. controls (0,-32) and (22.5,-32) .. (22.5,-20) -- (22.5,6);
	\draw[ultra thick] (5,-5) -- (5,-10) .. controls (5,-25) and (17.5,-25) .. (17.5,-20) -- (17.5,6);
	
	\draw[ultra thick] (17.5,20) -- (17.5,20) .. controls (17.5,15) and (22.5,15) .. (22.5,20) -- (22.5,20);
	\draw[ultra thick] (20,10) -- (20,16);
	\draw[ultra thick] (17.5,6) -- (17.5,6) .. controls (17.5,11) and (22.5,11) .. (22.5,6) -- (22.5,6);
	
	\draw[ultra thick] (11,-15) -- (11,-15) .. controls (8,-12) and (10,-7) .. (12,-5) -- (12,-5);
	\draw[ultra thick] (27,-7) -- (27,-7) .. controls (27,-2) and (21,5) .. (12,-5) -- (12,-5);
	\fill[color=white] (18,-17.2) circle (2.5);
	\fill[color=white] (22,-14.8) circle (2.5);
	\fill[color=white] (17.7,-1.8) circle (2.5);
	\fill[color=white] (22,0) circle (2.5);
	\fill[color=white] (20,0) circle (2.5);
	\draw[ultra thick] (22.5,-10) -- (22.5,6);
	\draw[ultra thick] (17.5,-5) -- (17.5,6);
	\draw[ultra thick] (27,-7) -- (27,-7) .. controls (27,-12) and (17,-22) .. (11,-15) -- (11,-15);
	}
= \sum_{\beta} \frac{1}{|\sigma|} S_{\beta,k}
\tikzmath{
	\draw[ultra thick] (5,5) -- (5,5) .. controls (5,0) and (0,0) .. (0,-5) -- (0,-5);
	\fill[color=white] (2.5,0) circle (2); 
	\draw[ultra thick] (0,5) -- (0,5) .. controls (0,0) and (5,0) .. (5,-5) -- (5,-5);
	\draw (-2,5) node {$i$};
	\draw (7,5) node {$i$};
	\draw (30,-3) node {$\beta$};	
	
	\draw[ultra thick] (0,5) -- (0,10) .. controls (0,32) and (22.5,32) .. (22.5,20) -- (22.5,6);
	\draw[ultra thick] (5,5) -- (5,10) .. controls (5,25) and (17.5,25) .. (17.5,20) -- (17.5,6);
	\draw[ultra thick] (0,-5) -- (0,-10) .. controls (0,-32) and (22.5,-32) .. (22.5,-20) -- (22.5,6);
	\draw[ultra thick] (5,-5) -- (5,-10) .. controls (5,-25) and (17.5,-25) .. (17.5,-20) -- (17.5,6);
	
	\draw[ultra thick] (11,-10) -- (11,-10) .. controls (8,-7) and (10,-2) .. (12,0) -- (12,0);
	\draw[ultra thick] (27,-2) -- (27,-2) .. controls (27,3) and (21,10) .. (12,0) -- (12,0);
	\fill[color=white] (18,-12.2) circle (2.5);
	\fill[color=white] (22,-9.8) circle (2.5);
	\fill[color=white] (17.7,4.8) circle (2.5);
	\fill[color=white] (22,5) circle (2.5);
	\draw[ultra thick] (22.5,-5) -- (22.5,11);
	\draw[ultra thick] (17.5,0) -- (17.5,11);
	\draw[ultra thick] (27,-2) -- (27,-2) .. controls (27,-7) and (17,-17) .. (11,-10) -- (11,-10);
	}
$$
where the equality comes from summing over $\sum_{f,\alpha} t^{f}_\alpha\circ (t^{f}_\alpha)^* = 1_{a_{\overline i} \times a_{\overline i}}$. Expanding the killing ring we obtain
\vspace{-4mm}
$$
= \sum_{\beta} \frac{1}{|\sigma|} S_{\beta,k}
\tikzmath{
	\draw[ultra thick] (5,5) -- (5,5) .. controls (5,0) and (0,0) .. (0,-5) -- (0,-5);
	\fill[color=white] (2.5,0) circle (2); 
	\draw[ultra thick] (0,5) -- (0,5) .. controls (0,0) and (5,0) .. (5,-5) -- (5,-5);
	\draw (-2,5) node {$i$};
	\draw (7,5) node {$i$};
	\draw (50,15) node {$\beta$};
	\draw (14.5,9) node {$\overline i$};
	\draw (26,9) node {$\overline \beta$};

	\draw[ultra thick] (0,5) -- (0,10) .. controls (0,32) and (32.5,32) .. (32.5,21) -- (32.5,21);
	\draw[ultra thick] (5,5) -- (5,10) .. controls (5,25) and (17.5,25) .. (17.5,15) -- (17.5,10);
	\draw[ultra thick] (17.5,0) -- (17.5,0) .. controls (17.5,5) and (22.5,5) .. (22.5,10) -- (22.5,10);
	\fill[color=white] (20,5) circle (2);
	\draw[ultra thick] (17.5,10) -- (17.5,10) .. controls (17.5,5) and (22.5,5) .. (22.5,0) -- (22.5,0);
	\draw[ultra thick] (22.5,0) -- (22.5,0) .. controls (22.5,-5) and (17.5,-5) .. (17.5,-10) -- (17.5,-10);
	\fill[color=white] (20,-5) circle (2);
	\draw[ultra thick] (22.5,-10) -- (22.5,-10) .. controls (22.5,-5) and (17.5,-5) .. (17.5,0) -- (17.5,0);
	
	\draw[ultra thick] (0,-5) -- (0,-10) .. controls (0,-32) and (32.5,-32) .. (32.5,-21) -- (32.5,-21);
	\draw[ultra thick] (5,-5) -- (5,-10) .. controls (5,-25) and (17.5,-25) .. (17.5,-15) -- (17.5,-10);
	
	\draw[ultra thick] (22.5,-10) -- (22.5,-10) .. controls (22.5,-18) and (37.5,-13) .. (42.5,-5) -- (42.5,-5);
	\draw[ultra thick] (22.5,10) -- (22.5,10) .. controls (22.5,18) and (40,23) .. (45,15) -- (45,15);
	\fill[color=white] (32.5,18.5) circle (2.5);
	\draw[ultra thick] (42.5,-5) -- (42.5,-5) .. controls (47.5,0) and (50,10) .. (45,15) -- (45,15);
	\draw[ultra thick] (32.5,-10) -- (32.5,21);
	\draw[ultra thick] (32.5,-21) -- (32.5,-15);
	}
= \sum_{\beta, \gamma, g} \frac{1}{|\sigma|} S_{\beta,k} \frac{\omega_\gamma}{\omega_i \omega_\beta}
\tikzmath{
	\draw[ultra thick] (5,5) -- (5,5) .. controls (5,0) and (0,0) .. (0,-5) -- (0,-5);
	\fill[color=white] (2.5,0) circle (2); 
	\draw[ultra thick] (0,5) -- (0,5) .. controls (0,0) and (5,0) .. (5,-5) -- (5,-5);
	\draw (-2,5) node {$i$};
	\draw (7,5) node {$i$};
	\draw (51,0) node {$\beta$};
	\draw (14.5,14) node {$\overline i$};
	\draw (26,8.5) node {$\overline \beta$};
	\draw (22.5,1.5) node {$\gamma$};
	\draw (17,1.5) node {$g$};	
	\draw (29,-13.5) node {$\overline \beta$};
	\draw (40.2,-12) node {$\overline i$};
	
	\draw[ultra thick] (0,5) -- (0,10) .. controls (0,32) and (37.5,42) .. (37.5,21) -- (37.5,21);
	\draw[ultra thick] (5,5) -- (5,10) .. controls (5,25) and (17.5,25) .. (17.5,15) -- (17.5,10);
	\draw[ultra thick] (5,-5) -- (5,-5) .. controls (5,-15) and (17.5,-15) .. (17.5,-5) -- (17.5,-5);

	\draw[ultra thick] (17.5,10) -- (17.5,10) .. controls (17.5,5) and (22.5,5) .. (22.5,10) -- (22.5,10);
	\draw[ultra thick] (19.8,-1) -- (19.8,6);
	\draw[ultra thick] (17.5,-5) -- (17.5,-5) .. controls (17.5,0) and (22.5,0) .. (22.5,-5) -- (22.5,-5);
	\draw[ultra thick] (22.5,-5) -- (22.5,-5) .. controls (22.5,-10) and (32.5,-6) .. (32.5,-11) -- (32.5,-11);
	\draw[ultra thick] (22.5,10) -- (22.5,10) .. controls (22.5,20) and (47.5,10) .. (47.5,0) -- (47.5,0);
	\draw[ultra thick] (37.5,-41) -- (37.5,-41) .. controls (37.5,-51) and (47.5,-51) .. (47.5,-41) -- (47.5,-41);
	\fill[color=white] (37.2,11.3) circle (2.5);
	\draw[ultra thick] (37.5,21) -- (37.5,-11);
	\draw[ultra thick] (47.5,-41) -- (47.5,0);
	
	\draw[ultra thick] (0,-5) -- (0,-10) .. controls (0,-32) and (25.5,-52) .. (32.5,-41) -- (32.5,-41);
	\draw[ultra thick] (32.5,-41) -- (32.5,-41) .. controls (32.5,-36) and (37.5,-36) .. (37.5,-31) -- (37.5,-31);
	\fill[color=white] (35,-36) circle (2);
	\draw[ultra thick] (37.5,-21) -- (37.5,-21) .. controls (37.5,-16) and (32.5,-16) .. (32.5,-11) -- (32.5,-11);
	\fill[color=white] (35,-16) circle (2);
	\draw[ultra thick] (32.5,-21) -- (32.5,-21) .. controls (32.5,-16) and (37.5,-16) .. (37.5,-11) -- (37.5,-11);
	\draw[ultra thick] (32.5,-31) -- (32.5,-31) .. controls (32.5,-26) and (37.5,-26) .. (37.5,-21) -- (37.5,-21);
	\fill[color=white] (35,-26) circle (2);
	\draw[ultra thick] (37.5,-31) -- (37.5,-31) .. controls (37.5,-26) and (32.5,-26) .. (32.5,-21) -- (32.5,-21);

	\draw[ultra thick] (37.5,-41) -- (37.5,-41) .. controls (37.5,-36) and (32.5,-36) .. (32.5,-31) -- (32.5,-31);
	}.
$$
This follows from the formula \cite[\Eq (2.30)]{Reh90} for the coefficients of the monodromy, which are invariant and depend only on modular data (phases). Similarly we get
\vspace{-5mm}
$$
= \sum_{\beta, \gamma, g, \eta, h} \frac{1}{|\sigma|} S_{\beta,k} \frac{\omega_\gamma \omega_\eta}{\omega_i^2 \omega_\beta^2}
\tikzmath{
	\draw[ultra thick] (5,5) -- (5,5) .. controls (5,0) and (0,0) .. (0,-5) -- (0,-5);
	\fill[color=white] (2.5,0) circle (2); 
	\draw[ultra thick] (0,5) -- (0,5) .. controls (0,0) and (5,0) .. (5,-5) -- (5,-5);
	\draw (-2,5) node {$i$};
	\draw (7,5) node {$i$};
	\draw (51,0) node {$\beta$};
	\draw (22.5,1.5) node {$\gamma$};
	\draw (17,1.5) node {$g$};	
	\draw (24,-14.5) node {$\overline \beta$};
	\draw (38,-28.5) node {$\eta$};
	\draw (32,-27.5) node {$h$};		
	
	\draw[ultra thick] (0,5) -- (0,10) .. controls (0,32) and (37.5,42) .. (37.5,21) -- (37.5,21);
	\draw[ultra thick] (5,5) -- (5,10) .. controls (5,25) and (17.5,25) .. (17.5,15) -- (17.5,10);
	\draw[ultra thick] (5,-5) -- (5,-5) .. controls (5,-15) and (17.5,-15) .. (17.5,-5) -- (17.5,-5);
	
	\draw[ultra thick] (17.5,10) -- (17.5,10) .. controls (17.5,5) and (22.5,5) .. (22.5,10) -- (22.5,10);
	\draw[ultra thick] (19.8,-1) -- (19.8,6);
	\draw[ultra thick] (17.5,-5) -- (17.5,-5) .. controls (17.5,0) and (22.5,0) .. (22.5,-5) -- (22.5,-5);
	\draw[ultra thick] (22.5,10) -- (22.5,10) .. controls (22.5,20) and (47.5,10) .. (47.5,0) -- (47.5,0);
	\draw[ultra thick] (37.5,-36) -- (37.5,-36) .. controls (37.5,-46) and (47.5,-46) .. (47.5,-36) -- (47.5,-36);
	\fill[color=white] (37.2,11.3) circle (2.5);
	\draw[ultra thick] (37.5,21) -- (37.5,-11);
	\draw[ultra thick] (47.5,-36) -- (47.5,0);
	
	\draw[ultra thick] (0,-5) -- (0,-10) .. controls (0,-32) and (25.5,-47) .. (32.5,-36) -- (32.5,-36);
	\draw[ultra thick] (37.5,-21) -- (37.5,-21) .. controls (37.5,-16) and (22.5,-10) .. (22.5,-5) -- (22.5,-5) ;
	\fill[color=white] (34.5,-16.5) circle (2.3);
	\draw[ultra thick] (32.5,-21) -- (32.5,-21) .. controls (32.5,-16) and (37.5,-16) .. (37.5,-11) -- (37.5,-11);
	\draw[ultra thick] (32.5,-21) -- (32.5,-21) .. controls (32.5,-26) and (37.5,-26) .. (37.5,-21) -- (37.5,-21);
	\draw[ultra thick] (35,-25) -- (35,-32);
	\draw[ultra thick] (32.5,-36) -- (32.5,-36) .. controls (32.5,-31) and (37.5,-31) .. (37.5,-36) -- (37.5,-36);
	}
$$
where $\beta,\gamma,\eta$ run over the spectrum of $\C$ and $g,h$ run over bases of isometries. The crucial step is to rewrite the previous (by naturality and multiplicativity of the braiding) as
$$
= \sum_{\beta, \gamma, g, \eta, h} \frac{1}{|\sigma|} S_{\beta,k} \frac{\omega_\gamma \omega_\eta}{\omega_i^2 \omega_\beta^2}
\tikzmath{
	\draw[ultra thick] (17.5,20) -- (17.5,20) .. controls (17.5,15) and (22.5,15) .. (22.5,20) -- (22.5,20);
	\draw[ultra thick] (19.8,9) -- (19.8,16);
	\draw[ultra thick] (17.5,5) -- (17.5,5) .. controls (17.5,10) and (22.5,10) .. (22.5,5) -- (22.5,5);
	\draw (22.5,12) node {$\gamma$};
	\draw (17,12) node {$g$};	
	\draw (22.5,-12) node {$\eta$};
	\draw (17,-12) node {$h$};	
	\draw (0,0) node {$i$};	
	\draw (26,0) node {$\overline \beta$};	
	\draw (45,0) node {$\beta$};	
	\draw (-6,0) node {$i$};	
	
	\draw[ultra thick] (17.5,5) -- (17.5,5) .. controls (17.5,-3) and (9,-5) .. (7.5,-5) -- (7.5,-5);
	\fill[color=white] (16.2,-0.2) circle (2.1);
	\draw[ultra thick] (7.5,-5) -- (7.5,-5) .. controls (7.5,-5) and (2.5,-5) .. (2.5,0) -- (2.5,0);
	\draw[ultra thick]  (2.5,0) -- (2.5,0) .. controls  (2.5,5) and (7.5,5) .. (7.5,5) -- (7.5,5);
	\draw[ultra thick] (7.5,5) -- (7.5,5) .. controls (9,5) and (17.5,3) .. (17.5,-5) -- (17.5,-5); 
	\draw[ultra thick] (22.5,5) -- (22.5,-5);
	
	\draw[ultra thick] (17.5,-20) -- (17.5,-20) .. controls (17.5,-15) and (22.5,-15) .. (22.5,-20) -- (22.5,-20);
	\draw[ultra thick] (19.8,-9) -- (19.8,-16);
	\draw[ultra thick] (17.5,-5) -- (17.5,-5) .. controls (17.5,-10) and (22.5,-10) .. (22.5,-5) -- (22.5,-5);
	
	\draw[ultra thick] (17.5,20) -- (17.5,20) .. controls (17.5,30) and (-3.5,30) .. (-3.5,20) -- (-3.5,20);
	\draw[ultra thick] (22.5,20) -- (22.5,20) .. controls (22.5,30) and (42,30) .. (42,20) -- (42,20); 
	\draw[ultra thick] (17.5,-20) -- (17.5,-20) .. controls (17.5,-30) and (-3.5,-30) .. (-3.5,-20) -- (-3.5,-20);
	\draw[ultra thick] (22.5,-20) -- (22.5,-20) .. controls (22.5,-30) and (42,-30) .. (42,-20) -- (42,-20);
	\draw[ultra thick] (-3.5,20) -- (-3.5,-20);
	\draw[ultra thick] (42,20) -- (42,-20);
	}
= \sum_{\beta, \gamma, g} \frac{1}{|\sigma|} S_{\beta,k} \frac{\omega_\gamma^2}{\omega_i \omega_\beta^2}
\tikzmath{
	\draw[ultra thick] (0,-5) -- (0,10) .. controls (0,32) and (22.5,32) .. (22.5,10) -- (22.5,10);
	\draw[ultra thick] (5,-5) -- (5,10) .. controls (5,25) and (17.5,25) .. (17.5,10) -- (17.5,10);
	\draw[ultra thick] (0,-5) -- (0,-10) .. controls (0,-32) and (22.5,-32) .. (22.5,-10) -- (22.5,-10);
	\draw[ultra thick] (5,-5) -- (5,-10) .. controls (5,-25) and (17.5,-25) .. (17.5,-10) -- (17.5,-10);
	\draw (-3,0) node {$\beta$};
	\draw (7.5,0) node {$i$};	
	\draw (23,0) node {$\gamma$};	
	\draw (17,3.5) node {$g$};	
	\draw (17,-4) node {$g$};	
	
	\draw[ultra thick] (17.5,10) -- (17.5,10) .. controls (17.5,5) and (22.5,5) .. (22.5,10) -- (22.5,10);
	\draw[ultra thick] (20,6) -- (20,-6);
	\draw[ultra thick] (17.5,-10) -- (17.5,-10) .. controls (17.5,-5) and (22.5,-5) .. (22.5,-10) -- (22.5,-10);
	} 
$$
where we have used $\omega_{\overline i} = \omega_i$ (the phases are a tortile unitary twist on $\C$, see \cite[\Def 2.3]{Mue00}) and $(t^h_\eta)^* \circ t^g_\gamma = \delta_{\gamma,\eta}\delta_{g,h}1_\gamma$. By the trace property of the left inverses we continue as
$$
= \sum_{\beta, \gamma, g} \frac{1}{|\sigma|} S_{\beta,k} \frac{\omega_\gamma^2}{\omega_i \omega_\beta^2} d_{\gamma}
\tikzmath{
	\draw[ultra thick] (2.5,9) -- (2.5,15);
	\draw[ultra thick] (2.5,-9) -- (2.5,-15);
	\draw (5,12) node {$g$};
	\draw (5.1,-11.4) node {$g$};
	\draw (2.5,19) node {$\overline \gamma$};
	\draw (2.5,-19) node {$\overline \gamma$};

	\draw[ultra thick] (0,5) -- (0,5) .. controls (0,10) and (5,10) .. (5,5) -- (5,5);
	\draw[ultra thick] (0,-5) -- (0,-5).. controls (0,-10) and (5,-10) .. (5,-5) -- (5,-5);
	\draw (-3.2,0) node {$\beta$};
	\draw (7.5,0) node {$i$};
	\draw[ultra thick] (0,5) -- (0,-5);
	\draw[ultra thick] (5,5) -- (5,-5);
	}
= \sum_{\beta, \gamma} \frac{1}{|\sigma|} S_{\beta,k} \frac{\omega_\gamma^2}{\omega_i \omega_\beta^2} d_{\gamma} N^{\overline \gamma}_{\beta, i}
$$
because $N^{\overline \gamma}_{\beta, i}$ is defined as the multiplicity of $[a_{\overline \gamma}]$ in $[a_{\beta}] \times [a_i]$. 
Moreover, $N^{\overline \gamma}_{\beta, i} = N^{i}_{\overline \beta, \overline \gamma}$ by Frobenius reciprocity, hence
$$ \Tr_{a_k}(\eps_{a_i,a_i}) 
= \sum_{\beta, \gamma} \frac{1}{|\sigma|} S_{\beta,k} \frac{\omega_\gamma^2}{\omega_i \omega_\beta^2} d_{\gamma} N^{i}_{\overline \beta, \overline \gamma} 
= \omega_i^{-1}\sum_{\beta, \gamma} \overline{S_{\beta,k}} S_{\gamma,0} \frac{\omega_\gamma^2}{\omega_\beta^2} N^{i}_{\beta, \gamma}$$
after changing the names of the summation indices, using $S_{\overline \beta, k} = \overline{S_{\beta, k}}$ and the definition of $S_{\gamma,0}$.

In the case $m=0$ we have $\tau_{k,i} = 0$ by definition, and the only element in $\Hom_{\C}(a_k,a_i \times a_i)$ is the zero morphism. Anyway one can repeat the proof with $t_k = 0$, which is soon absorbed in the summation exploiting modularity of $\C$, \ie, $S^2 = C$. So we have shown (\ref{eq:bantaytrace}) in either case. 

The proof of (\ref{eq:bantayFSindicator}) now follows by noticing that $S_{\beta,0}$ is real, because $k=0$ is self-conjugate, and we are done.
\end{proof}

While to our knowledge there is no proof in any context of equation
(\ref{eq:bantaytrace}) for the traces of self-braiding
intertwiners, the equation (\ref{eq:bantayFSindicator}) for the Frobenius-Schur indicator has already been considered and derived in more abstract tensor categorical and Hopf-algebraic settings. See \cite[\Thm 7.5, \Rmk 7.6]{NgSc07} for its first appearance in the context of MTCs, \cite[\Thm 4.25]{Wan10lec} for a proof using string-diagrammatical calculus, \cite{Ng12} for an overview on Frobenius-Schur indicators in Hopf algebras, and \cite[\Thm VI.1.3]{BruPhD} for a generalization to self-conjugate objects in premodular categories.

\begin{remark}
It is interesting to notice that the $S$ and $T$ matrices employed in
\cite{Ban97} are those arising from the modular transformations of the
Virasoro characters in CFT, see \cite{Ver88}. On the other hand, our
present proof holds for UMTCs, where the $S$ and $T$ matrices are
defined using left inverses and braidings, as in \cite{Reh90}. Both
versions of $S$, $T$ represent the modular group and diagonalize the
fusion rules. The equality of the expressions appearing in either case
(up to Remark \ref{rmk:conjugationmissing}) is another hint in the
direction that the two versions should coincide, despite a proof 
of the very existence of a modular transformation law for characters
in general rational CFT is still missing.
\end{remark}

\begin{remark}
We also stress that Proposition \ref{prop:bantaytrace} expresses
characteristic features of UMTCs rather than rational CFTs, and that
they should be attributed to the former, \eg, when classification
issues are concerned. It is well-known that different CFTs can give
rise to equivalent UMTCs (as abstract braided tensor categories), \eg,
by taking tensor products with \lq holomorphic' nets. This \lq
degeneration' problem is investigated in \cite{GiRe15} in the language
of Algebraic QFT, where a clear distinction between rational CFTs
(described as Haag-Kastler nets) and the associated UMTCs (given by
the respective categories of DHR representations) can be made. In
AQFT, rationality = modularity of the representation category is a consequence of
natural structure assumptions on the local observables \cite[\Cor 37]{KLM01}. 
\end{remark}

\section{Applications}

\subsection{On the realizability of modular data}\label{subsec:realiz}

As observed by Bantay in the case of primary fields in rational CFT, see comments after \cite[\Eq (5)]{Ban97}, the formula for the trace of self-braidings (\ref{eq:bantaytrace}) we derived for UMTCs imposes constraints on the admissible modular data. Knowing the dimension $N_{i,i}^{k}$ of the matrix of coefficients of $\eps_{a_i,a_i}$ in the channel corresponding to $a_k \prec a_i \times a_i$ and its trace $\tau_{k,i}$, one can compute the \emph{multiplicity} $m_{k,i}^{\pm}$ of either eigenvalue $\pm \omega_i^{-1}\omega_k^{1/2}$ (where $\pm$ depends also on the choice of the square root), namely
\begin{equation}\label{eq:eigenmult}
m_{k,i}^{\pm} = \frac{1}{2} \left( N^{k}_{i,i} \pm \omega_i\omega_k^{-1/2}\tau_{k,i} \right)
\end{equation}
\cf \cite[\Eq (4)]{Ban97}. In particular, the number $\omega_i\omega_k^{-1/2}\tau_{k,i}$ must be an integer with the same parity of $N^{k}_{i,i}$ for every $i,k = 0,\ldots,n$. Moreover, it must not exceed the range $-N^{k}_{i,i}, \ldots, N^{k}_{i,i}$. Thus we can add another item in the list of constraints \cite[\Def 2.14]{BNRW15} which modularity of $\C$ imposes on its modular data, at least in the unitary case.

\begin{remark}
The property that $\nu_{i}$ takes values in $\{0,\pm 1\}$ is independent of all the relations among the entries of the modular data, see \cite[\Sec 2]{Gan05}. A fortiori the integrality properties on products of $\omega_i$ and $\tau_{k,i}$ listed above are independent as well. 
\end{remark}

\subsection{On the classification of UMTCs}\label{subsec:classif}

We want to address the question whether the modular data of a UMTC $\C$ uniquely determine its unitary braided tensor equivalence class, \ie, completeness of the set of numerical invariants. The answer is expected to be positive among experts, see, e.g., \cite{RSW09}.

In this note, using Proposition \ref{prop:bantaytrace}, we show that in a UMTC $\C$ the $R$-matrices (see \cite[\Sec 3]{FRS92}, \cite[\Sec 3.3]{DHW13}) can be \lq canonically' expressed in terms of the modular data, namely 

\begin{proposition}\label{prop:Rcan}
Let $\C$ be a UMTC with modular data $\{n,\Delta,N,S,T\}$. Choose representatives $a,b,c,\ldots$ in each unitary isomorphism class of irreducible objects in $\Delta$. 

There is a suitable choice of orthonormal bases of isometries $t^{e}_{c,ab}$ in $\Hom_{\C}(c, a \times b)$, where  $e=1,\ldots, N^{k}_{i,j}$ and $i,j,k$ label respectively the sector of $a,b,c$ in $\Delta$, for every triple $(a,b,c)$, such that the following holds. 

If $a\neq b$, then
\begin{equation}\label{eq:Rabcan}
R_{c,ab} = R_{c,ba} = \Big(\frac{\omega_k}{\omega_i\omega_j}\Big)^{1/2}\oneop_{N^{k}_{i,j}}\, ,
\quad R^\op_{c,ab} = R^\op_{c,ba} =\Big(\frac{\omega_i\omega_j}{\omega_k}\Big)^{1/2}\oneop_{N^{k}_{i,j}}
\end{equation}
where $\oneop_{N^{k}_{i,j}}$ is the identity of $M_{N^{k}_{i,j}\times N^{k}_{i,j}}(\CC)$, $R_{c,ab}$ and $R^\op_{c,ab}$ are respectively the $R$-matrices of the braiding $\eps$ and its opposite $\eps^\op$.

If $a=b$, then
\begin{equation}\label{eq:Raacan}
R_{c,aa} = \omega_i^{-1}\omega_k^{1/2} (E^{+} - E^{-})\, , \quad R^\op_{c,aa} = \omega_i\omega_k^{-1/2} (E^{+} - E^{-})
\end{equation}
where $E^{\pm}$ are orthogonal projections in $M_{N^{k}_{i,j}\times N^{k}_{i,j}}(\CC)$ $($the eigenprojections of $\eps_{a,a}$ in $\Hom_{\C}(c, a \times a)$$)$ such that $E^{+} E^{-} = 0$ and $E^{+} + E^{-} = \oneop_{N^{k}_{i,i}}$, whose dimensions are fixed by the modular data.
\end{proposition}

\begin{proof}
By definition, the entries of $R_{c,ab}$ and $R^\op_{c,ab}$ are
$$(R_{c,ab})^{e,f} := (t^f_{c,ba})^*\circ \eps_{a,b} \circ t^e_{c,ab}\, , \quad (R^\op_{c,ab})^{e,f} := (t^f_{c,ba})^*\circ \eps^\op_{a,b} \circ t^e_{c,ab}$$
where $t^e_{c,ab}$ and $t^f_{c,ba}$ belong to \emph{some} choice of orthonormal bases of isometries, and they give rise to unitary matrices in $M_{N^{k}_{i,j}\times N^{k}_{i,j}}(\CC)$, because $N^{k}_{i,j} = N^{k}_{j,i}$. Moreover,
$$\eps_{a,b} = \sum_{[c],e,f}  (R_{c,ab})^{e,f}\, t^f_{c,ba}\circ (t^e_{c,ab})^*,\quad \eps^\op_{a,b} = \sum_{[c],e,f}  (R^\op_{c,ab})^{e,f}\, t^f_{c,ba}\circ (t^e_{c,ab})^*.$$
Due to $\eps^\op_{b,a}\circ\eps_{a,b} = \oneop_{a\times b}$ and to the formula for the coefficients of the monodromy \cite[\Eq (2.30)]{Reh90}, namely
\begin{equation}\label{eq:monodrexpansion}
\eps_{b,a}\circ\eps_{a,b} = \frac{1}{\omega_i \omega_j} \sum_{k\in\Delta} \omega_k \sum_{e=1,\ldots,N^{k}_{i,j}} t^e_{c,ab}\circ (t^e_{c,ab})^*
\end{equation}
where $i,j,k$ label respectively the sector of $a,b,c$ in $\Delta$, we have $R^\op_{c,ba} = (R_{c,ab})^{-1}$, $R_{c,ba} = \frac{\omega_k}{\omega_i \omega_j}(R_{c,ab})^{-1}$ and $R^\op_{c,ab} = \frac{\omega_i \omega_j}{\omega_k} R_{c,ab}$. In particular, there is only one independent $R$-matrix for every triple $(a,b,c)$, namely $R_{c,ab}$, irrespectively of the choice of the bases of isometries.

If $a\neq b$ we can choose $t^e_{c,ab}$ and $t^f_{c,ba}$ independently and in such a way that $R_{c,ab}$ is diagonal and equal to the scalar matrix 
$(\frac{\omega_k}{\omega_i\omega_j})^{1/2}\oneop_{N^{k}_{i,j}}$, thus $R_{c,ab} = R_{c,ba} = (R^\op_{c,ab})^{-1} = (R^\op_{c,ba})^{-1} = (\frac{\omega_k}{\omega_i\omega_j})^{1/2}\oneop_{N^{k}_{i,j}}$ and (\ref{eq:Rabcan}) is proved.

The situation is more complicated when $a=b$. In that case we can choose a basis $t^e_{c,aa}$ which diagonalizes $R_{c,aa}$ and by (\ref{eq:monodrexpansion}) we know that the only two possible eigenvalues are $\pm\omega_i^{-1}\omega_k^{1/2}$ where $i$ and $k$ label respectively the sector of $a$ and $c$ in $\Delta$. There is, however, no a priori choice of the basis of isometries in $\Hom_\C(c,a\times a)$ which fixes the sign ambiguity. But now, we know by equation (\ref{eq:bantaytrace}) of Proposition \ref{prop:bantaytrace} that the trace of $\eps_{a,a}$ in the channel $c\prec a \times a$ is an invariant of the UMTC which depends only on modular data, and the same is true for the dimension $N^{k}_{i,i}$ of the matrix $R_{c,aa}$. Hence the multiplicities of eigenvalues are determined by modular data, as in \Sec \ref{subsec:realiz}, and a suitable permutation of the isometries gives (\ref{eq:Raacan}), concluding the proof.
\end{proof}

It is well known that the simultaneous knowledge of the (braiding) $R$-matrices and of the (fusion) $F$-matrices, in some choice of bases, completely determines a UMTC up to unitary braided tensor equivalence, see \cite[\Prop 3.12]{DHW13} for a detailed proof. In view of Proposition \ref{prop:Rcan} we can make the following observation:

\begin{remark}
Let $\C$ be a UMTC with modular data $\{n,\Delta,N,S,T\}$ and choose
bases of isometries in $\Hom_\C(c,a\times b)$ such that the
$R$-matrices $\{R_{c,ab}\}$ are given as in Proposition
\ref{prop:Rcan}. Then every other UMTC with the same modular data as
$\C$ (we know by \cite{ENO05} that there are finitely many candidates
up to unitary braided tensor equivalence) arises as a solution of the
polynomial equations $FF=FFF$, $FRF=RFR$, $FR^{-1}F=R^{-1}FR^{-1}$
(omitting indices, see \cite{FRS92}, \cite{DHW13}). The first set of
equations corresponds to the pentagonal diagrams defining the tensor
structure, the second and third to the hexagonal diagrams defining the
braiding (also known as braiding-fusion equations \footnote{Some
  authors, \eg, \cite{MoSe88}, use instead braiding matrices $B$
  depending on six sectors, that are related to $R$ (schematically) by $B = F^{-1}RF$,
such that the braiding-fusion relations take the form $FB=BBF$. The present $R$ matrices encode only the independent information about the braiding
beyond the fusion.}). The latter provide a system of polynomial \emph{constraints} on the set of possible tensor structures (specified by $F$) that are compatible with the modular data (which determine $R$ in the sense of Proposition \ref{prop:Rcan}).
\end{remark}

Now, let $\C_1$ and $\C_2$ be two UMTCs having the same modular data and the same (or equivalent) underlying unitary fusion structure. Again in view of Proposition \ref{prop:Rcan} it is natural to ask if they are necessarily equivalent as UMTCs. 

As a first step, using the information on the \lq spectrum' of the braiding contained in the modular data, we can say the following

\begin{proposition}\label{prop:delta2U=alpha}
Let $\C_1$ and $\C_2$ be two UMTCs with the same underlying strict UFTC structure, \ie, $\C_1 = (\C,\times,\id, \eps)$, $\C_2 = (\C,\times,\id, \widetilde\eps)$, where $\times$ and $\id$ denote respectively the tensor multiplication and tensor unit. Assume that $\C_1$ and $\C_2$ have the same modular data. Then $\C$ can be equipped with another (equivalent, but non-strict) UFTC structure $(\C,\times,\alpha, \id)$ having tensor multiplication $\times$ and associator $\alpha$, where the equivalence is of the form $(\Id, U, \oneop)$.
Moreover, $\widetilde \eps$ is a braiding on $(\C,\times,\alpha, \id)$, and $(\Id, U, \oneop) : (\C,\times, \id, \eps) \rightarrow (\C,\times, \alpha, \id, \widetilde\eps)$ is a unitary braided tensor equivalence.
\end{proposition}

\begin{proof}
Choose representatives $a,b,c,\ldots$ in each unitary isomorphism class of irreducible objects. As in the proof of Proposition \ref{prop:Rcan}, we know that the spectrum of the monodromies and self-braidings is fixed (including multiplicities) by the modular data. Hence for every pair $(a,b)$ we have unitaries $U_{a,b}\in\Hom_\C(a\times b, a\times b)$, $U_{b,a}\in\Hom_\C(b\times a, b\times a)$, $U_{a,a}\in\Hom_\C(a\times a, a\times a)$, $U_{b,b}\in\Hom_\C(b\times b, b\times b)$ such that
$$\widetilde\eps_{b,a}\circ\widetilde\eps_{a,b} = U^*_{a,b}\circ\eps_{b,a}\circ\eps_{a,b}\circ U_{a,b}\, , \quad 
\widetilde\eps_{a,b}\circ\widetilde\eps_{b,a} = U^*_{b,a}\circ\eps_{a,b}\circ\eps_{b,a}\circ U_{b,a}$$
$$\widetilde\eps_{a,a} = U^*_{a,a}\circ\eps_{a,a}\circ U_{a,a}\, , \quad
\widetilde\eps_{b,b} = U^*_{b,b}\circ\eps_{b,b}\circ U_{b,b}$$
uniquely determined up to left multiplication with unitaries that commute with the respective monodromy or self-braiding. We can arrange them in such a way that, in addition, $\widetilde\eps_{a,b} = U^*_{b,a}\circ\eps_{a,b}\circ U_{a,b}$\,, $\widetilde\eps_{b,a} = U^*_{a,b}\circ\eps_{b,a}\circ U_{b,a}$, \ie
\begin{equation}\label{eq:Uisbraided}
\xymatrix{
	a\times b \ar[r]^{U_{a,b}} \ar[d]_{\widetilde\eps_{a,b}}  &  a\times b \ar[d]^{\eps_{a,b}} \\ 
	b \times a \ar[r]_{U_{b,a}}  &  b \times a
}
\end{equation}
is a commutative diagram. The unitaries $U$ can be extended to arbitrary pairs of objects in $\C$, expressing them as direct sums of irreducibles in the previous choice of representatives and choosing orthonormal bases of isometries realizing the direct sums. It is easy to see that the $U$ are well-defined (independent of the choice of bases of isometries), unitary, natural and make the diagrams (\ref{eq:Uisbraided}) commute for every $a,b$ in $\C$.

As observed in \cite[\Lem 3.1]{Sch01} there is a unique UFTC structure on $\C$ such that $(\Id, U, \oneop):(\C,\times,\id)\rightarrow(\C,\times,\alpha, \id)$ is a unitary tensor equivalence, where the associator $\alpha$ is defined by the left vertical arrow in the diagrams below, such that 
\begin{equation}\label{eq:delta2U=alpha}
\xymatrix{
	(a\times b)\times c \ar[r]^{U_{a,b}\times \oneop_c} \ar[d]_{\alpha_{a,b,c}}  &  
	(a\times b)\times c \ar[r]^{U_{a\times b,c}}  	&
	(a\times b)\times c \ar[d]^{\oneop_{a,b,c}} \\ 
	a\times (b\times c) \ar[r]_{\oneop_a \times U_{b,c}}  &
	a\times (b\times c) \ar[r]_{U_{a,b\times c}}  &
	a\times (b\times c)
}
\end{equation}
commutes for every $a,b,c$ in $\C$, \ie, $\alpha_{a,b,c} := \oneop_a \times U^*_{b,c} \circ U^*_{a,b\times c} \circ U_{a\times b, c} \circ U_{a,b}\times\oneop_c$. In particular, $\alpha$ makes the pentagon diagrams commute. One can check directly that $(\C,\times, \alpha, \id, \widetilde\eps)$ is a braided category, and this proves the second statement.
\end{proof}

As a second step, one would like to understand up to which extent, in the assumptions of Proposition \ref{prop:delta2U=alpha}, the information on the $R$-matrices given by the modular data can be used to construct an actual braided tensor equivalence between $\C_1$ and $\C_2$. However, Ocneanu rigidity (\ie, the vanishing of the Davydov-Yetter cohomology, see \cite[\Sec 7]{ENO05}, \cite[\Sec E.6]{Kit06}), poses a serious obstruction in this direction, namely 

\begin{proposition}\label{prop:Ocneanurig}
Let $\C_1$ and $\C_2$ as in Proposition \ref{prop:delta2U=alpha}, and consider $\C_3 := (\C,\times, \alpha, \id, \widetilde\eps)$. Then there exists a unitary braided tensor equivalence between $\C_1$ and $\C_2$ of the form $(\Id,W,\oneop):\C_1\rightarrow\C_2$, or equivalently a unitary braided tensor equivalence $(\Id,V,\oneop):\C_2\rightarrow\C_3$, where $W = UV^*$, if and only if $\eps = \widetilde\eps$.
\end{proposition}

\begin{proof}
To prove the non-trivial implication, observe that $(\Id,V,\oneop)$ is a tensor equivalence if and only if $\alpha = \delta^2(V)$, where we denote $\delta^2(V)_{a,b,c} := \oneop_a \times V^*_{b,c} \circ V^*_{a,b\times c} \circ V_{a\times b, c} \circ V_{a,b}\times\oneop_c$. But also $\alpha = \delta^2(U)$ by definition and $\delta^2(VU^*) = \oneop$ by naturality. The latter is infinitesimally read as a 2-cocycle condition in the sense of the Davydov-Yetter, see \cite[\Sec 4, 5]{Dav97}, \cite[\Sec E.6.2]{Kit06}, but the cohomology of this complex vanishes \cite[\Thm 2.27]{ENO05}, hence
\begin{equation}\label{eq:2-coboundaryexpansion}
V_{a,b}U_{a,b}^* = \frac{1}{X_a X_b} \sum_{k\in\Delta} X_c \sum_{e=1,\ldots,N^{k}_{i,j}} t^e_{c,ab}\circ (t^e_{c,ab})^*
\end{equation}
for every $a,b$ in $\C$, where $X_a$ are phases, $i,j,k$ label respectively the sectors of $a,b,c$ in $\Delta$, and $t^{e}_{c,ab}$, $e=1,\ldots, N^{k}_{i,j}$, is a basis of isometries in $\Hom_{\C}(c, a \times b)$.

Now, $(\Id,V,\oneop)$ is braided if and only if $\widetilde\eps_{a,b} V_{a,b} = V_{b,a} \widetilde\eps_{a,b}$ for every $a,b$ in $\C$. By (\ref{eq:2-coboundaryexpansion}) this is equivalent to
$$\sum_{\substack{k\in\Delta, \\ e=1,\ldots,N^{k}_{i,j}}} X_c \, \widetilde\eps_{a,b} \circ t^e_{c,ab}\circ (t^e_{c,ab})^* \circ U_{a,b} = \sum_{\substack{l\in\Delta, \\ f=1,\ldots,N^{k}_{i,j}}} X_d\,  t^f_{d,ba}\circ (t^f_{d,ba})^* \circ U_{b,a} \circ \widetilde\eps_{a,b}$$
hence $(t^f_{d,ba})^* \circ \widetilde\eps_{a,b} \circ t^e_{c,ab} = (t^f_{d,ba})^* \circ \eps_{a,b} \circ t^e_{c,ab}$ and we have $\widetilde\eps_{a,b} = \eps_{a,b}$.
\end{proof}

On one hand, the results of this section say that two UMTCs $\C_1$ and $\C_2$ having the same modular data also have \lqq the same braiding", or better the same $R$-matrices in a suitable choice of orthonormal bases of isometries both in $\C_1$ and $\C_2$ (\Prop \ref{prop:Rcan}). On the other hand, if we fix the tensor structure and try to make the previous statement more functorial by constructing a unitary braided tensor equivalence (using the modular data via the trace formula of \Prop \ref{prop:bantaytrace}, with underlying functor the identity), it turns out that this is only possible when the two braidings actually coincide (\Prop \ref{prop:Ocneanurig}).

Roughly speaking, this pushes the classification problem of UMTCs by means of their modular data back to the question on how can the tensor structure itself be read off the modular matrices $S$, $T$.


\bigskip

{\bf Acknowledgements.}
Supported by the German Research Foundation (Deutsche Forschungsgemeinschaft (DFG)) through the Institutional Strategy of the University of G\"ottingen, and by the European Research Council (ERC) through the Advanced Grant QUEST \lq\lq Quantum Algebraic Structures and Models".
We are indebted to M. Bischoff for several discussions on this topic, and for his motivating interest, and to R. Longo and R. Conti for comments and questions about this work. L.G. thanks also P. Naaijkens for an invitation to Hannover, and for helpful conversations in that occasion together with his colleagues C. Brell and L. Fiedler.

\small


\def\cprime{$'$}

\end{document}